\documentclass[12pt,amstex]{amsart}

\usepackage{mathptmx}
\usepackage{mathrsfs}

\usepackage{verbatim}
\usepackage{url}
\usepackage[all]{xy}
\usepackage{color}

\usepackage[colorlinks=true,citecolor=blue]{hyperref}

\usepackage{stmaryrd}
\usepackage{epsfig}
\usepackage{amsmath}
\usepackage{amssymb}
\usepackage{appendix}
\usepackage{amscd}
\usepackage{graphicx}
\usepackage{pstricks}

\theoremstyle{plain}
\newtheorem{thm}{Theorem}[section]
\newtheorem{lem}[thm]{Lemma}

\newtheorem{prop}[thm]{Proposition}

\newtheorem{cor}[thm]{Corollary}
\theoremstyle{definition}
\newtheorem{de}[thm]{Definition}

\theoremstyle{remark}
\newtheorem{rem}[thm]{Remark}

\def \N {\mathbb N}

\def \id {{\rm id}}

\def \ep {\epsilon}

\DeclareMathOperator{\diam}{diam}

\begin{document}
	\title{On Feldman-Katok metric and entropy formulas}

\author{Fangzhou Cai}

\author{Jie Li$^*$}\let\thefootnote\relax\footnote{* Corresponding author. Research of Jie Li is supported by NNSF of China (Grant No. 12031019)}

\address[F. Cai]{School of Mathematics and Systems Science, Guangdong Polytechnic Normal University, Guangzhou, 510665, PR China}
\email{cfz@mail.ustc.edu.cn}
\address[J. Li]{School of Mathematics and Statistics, Jiangsu Normal University, Xuzhou, Jiangsu 221116, China}
\email{jiel0516@mail.ustc.edu.cn}

	\maketitle
	
	\begin{abstract}
		In this paper, we study the  Feldman-Katok metric and  give corresponding entropy formulas by replacing  Bowen metric with  Feldman-Katok metric. It turns out that the Feldman-Katok metric is  the weakest one, in the sense of allowing  desynchrony and jump  in the process of measuring orbit segments' distance,  that makes the entropy formulas valid.
Some related topics  are also discussed.
	\end{abstract}

	\section{Introduction}
	In ergodic theory a fundamental problem is to classify measure-preserving systems (MPS for short) up to isomorphism. In 1958, Kolmogorov \cite{K58} introduced the concept of entropy into ergodic theory and proved that entropy is an isomorphism invariant for MPSs. By calculating the entropies, Kolmogorov proved that the two-sided $(\frac{1}{2},\frac{1}{2})$-shift and the two-sided $(\frac{1}{3},\frac{1}{3},\frac{1}{3})$-shift are not isomorphic, since they had entropies $\log 2$ and $\log 3$ respectively.
	
	A remarkable
	achievement on the isomorphism problem is Ornstein’s theory \cite{orn}. He  proved that any two Bernoulli processes of equal entropy are isomorphic.
	In Ornstein’s theory the concept of
	 {\it finitely determined process} played an important role, the definition of which is based on the  Hamming distance $\bar{d_n}:$
	 $$\bar{d_n}(x_0x_1\ldots x_{n-1},y_0y_1\ldots y_{n-1})=\frac{|\{0\leq i\leq n-1:x_i\neq y_i\}|}{n},$$
	 where $x_0x_1\ldots x_{n-1}$ and $y_0y_1\ldots y_{n-1}$ are two finite sequences.
	
	In 1943, Kakutani introduced a notion of equivalence called {\it Kakutani equivalence} between ergodic MPSs\cite{kaku}.
	In 1976,  Feldman \cite{feld} introduced the notion of {\it loose Bernoullicity} and brought a new idea into the classification of MPSs. Then a theory which ran parallel to the theory of Ornstein’s was established \cite{katok,feld,orw}. It classified ergodic MPSs which are Kakutani equivalent to
	Bernoulli or Kronecker systems.
 In \cite{feld}, Feldman replaced $\bar{d}_n$ in the definition
	 of finitely determined process by the  {\it edit distance} $\bar{f}_n$:
	  $$\bar{f_n}(x_0x_1\ldots x_{n-1},y_0y_1\ldots y_{n-1})=1-\frac{k}{n},$$
	  where $k$ is the largest integer such that there exist $$0\leq i_1<\ldots<i_k\leq n-1, 0\leq j_1<\ldots<j_k\leq n-1$$ and $x_{i_s}=y_{j_s}$ for $s=1,\ldots,k.$
	  In this way, he defined
	 {\it finitely fixed process}, equivalently loosely Bernoulli system. Surprisingly, both Kronecker and Bernoulli
	 MPSs are  loosely Bernoulli systems. 	By Abramov’s formula \cite{ab}, if two ergodic MPSs are equivalent, then  either both of them have zero
	 entropy, or both have positive finite entropy, or both have infinite entropy. Hence we know that there are at least three
	 Kakutani equivalence classes of loosely Bernoulli systems as systems with zero,
	 positive but finite, and infinite entropy must belong to different classes. It turns out
	 that there are exactly three classes \cite{orw}.
	
	The Feldman-Katok metric was introduced in \cite{FK} as a topological counterpart of
	the edit distance $\bar{f}$ (see Definition \ref{def:fk-metric}).
	In \cite{loose},  the authors used the Feldman-Katok metric to characterize
	zero
	entropy loosely Bernoulli (the authors called it {\it loosely Kronecker})  MPSs and presented a purely topological characterization of their topological models.

	 \bigskip
	
	 Let $(X,d,T)$ (or $(X,T)$ for short) be a {\it topological
	 dynamical system} (TDS for short) in the sense that $T:X\to X$ is a continuous
	 map on the compact metric space $X$ with metric $d$.
	 Parallel to the measurable case, in 1965 Alder, Konheim and McAndrew \cite{AKM65} introduced an analogous notion of topological entropy for a TDS, as a topological conjucacy invariant.
	 In \cite{bowen}, Bowen gave an equivalent definition of topological entropy using the notion of spanning or separated sets.
	 The Bowen metric, which played a key role in his definition,
	  is defined as:
	 $$d_n(x,y)=\max_{0\leq i\leq n-1}d(T^ix,T^iy).$$
	 The Bowen ball is given by $B_n(x,\ep)=\{y\in X: d_n(x,y)<\ep\}.$
	 The link between topological entropy and measurable entropy is the well-known variational principle (see, e.g., \cite{Wal82}).
	 In \cite{katok80}, Katok introduced Bowen metric and Bowen ball into ergodic theory and gave an analogous measure-theoretic entropy formula. Following the idea, a local version of the entropy formula was also obtained \cite{bk}, which is known now as Brin-Katok formula.
	
	 Corresponding to the Hamming distance, the mean metric on a TDS $(X,T)$ is given by
	 $$
	 \bar{d_n}=\frac{1}{n}\sum_{i=0}^{n-1}d(T^ix,T^iy)
	 $$
	 for $x,y\in X$ and $n\in\N$. In recent years, mean metric is confirmed to be useful in studying mean equicontinuity, bounded complexity and its related issues like Sarnak conjecture, one may refer to the survey \cite{LYY21} for details. Gr\"{o}ger and J\"{a}ger \cite{GJ15} showed that the topologoical entropy defined by mean metric is equivalent to the classical topological entropy. Huang, Wang and Ye \cite{adv} proved that the Katok's entropy formula for ergodic measures is still valid when replacing the Bowen metric $d_n$ by the mean metric $\bar{d_n}$.  Huang, Chen and Wang \cite{chen} further proved the Katok's entropy formula and Brin-Katok formula of conditional entropy in mean metrics.

	 In this paper, we replace  Bowen metric $d_n$ by  Feldman-Katok metric  $d_{FK_n}$(Definition \ref{def:fk-metric}) and give analogous entropy formulas (Theorems \ref{thm:bowen-FK}, \ref{thm:Brin-Katok} and  \ref{thm:Katok-formula}). We note that when dealing with the Feldman-Katok metric, things may become  different or more complicated,  and so new techniques or finer estimates need to be developed (refer to Remarks \ref{rem:brin-katok:1}, \ref{rem:brin-katok:2} and \ref{rem:brin-katok:3}).

 A pseudometric induced by removing the order-perserving restriction in the definition of Feldman-Katok metric is also studied (see Section \ref{sect:weak-mean pseudometric}), and it turns out that such pseudometric is coincided with the weak-mean pseudometric introduced in \cite{zhen} (Theorem \ref{thm:FK=F}).  We explain that using weak-mean pseudometric to obtain entropy formula is hopeless (Theorem  \ref{thm:F-eq-weak}), which indicates that the Feldman-Katok metric is the weakest one, in the sense of  only requiring order preserving in the process of measuring orbit segments' distance, that makes the entropy formulas valid.
	
We also study measure-theoretic equicontinuity in Feldman-Katok metric case and strengthen some results in \cite{loose} (Section \ref{sect:FK-equi}). In particular, inspired by the main idea used in \cite[Theorem 4.5]{loose}, we give an alternative proof of \cite[Theorem 4.5]{loose}  (Theorem \ref{thm:f-bar} and Remark \ref{rem:fk-f-bar}),  which is helpful for us to understand how the Feldman-Katok metric works. Furthermore, we improve a result related to measure-theoretic discrete spectrum to non-ergodic case (Theorem \ref{thm:discrete-specturm}) and as a consequence get another new proof of the sufficiency of main result \cite[Theorem 4.5]{loose} (Corollary \ref{cor:discrete-specturm}).
 \medskip
	
	 It is worth noting that when considering entropy we get the same value among using the balls defined by $d_n$, $\bar{d_n}$ and $d_{FK_n}$, nevertheless, the advantage to use  $d_{FK_n}$ is that  it is an isomorphism invariant when studying the measure FK-complexity (Proposition \ref{prop:measure-complexity}) and it allows `time delay' and `space jump' by ignoring the synchronization of points  in measuring orbit segments' distance with only order preserving required, which is reasonable when studying statistical properties of long time behaviors and leads to the property, serving as the main source of differences between Feldman-Katok pseudometric $d_{FK}$ and the mean pseudometric $\bar{d}$, that $d_{FK}$ is invariant along the orbits (see, e.g., \cite[Fact 17]{FK}).
	
	 \medskip
	 The paper is organized as follows.
	In Section \ref{sect:2}, we give entropy formulas for Feldman-Katok metric. In Section \ref{sect:weak-mean pseudometric}, we study the weak-mean pseudometric related to Feldman-Katok metric. In Section \ref{sect:FK-equi}, we study $\mu$-$FK$-equicontinuity and strengthen some results in \cite{loose}. In Appendix we collect some proofs for completeness.

	\section{entropy formulas for Feldman-Katok metric}\label{sect:2}
	In this section we  give entropy formulas for Feldman-Katok metric.
	 We announce that all our basic definitions and notations (like invariant measure, entropy, isomorphism, etc.) are as in the textbook \cite{Wal82}, and we only explain some of them as we proceed.

	\medskip
	
	First we give the definition of Feldman-Katok metric \cite{loose}:
	
Let $(X,d,T)$ be a TDS. For $x, y\in X, \delta>0$ and $n\in\N$, we define an {\it $(n, \delta)$-match} of $x$ and $y$ to be an {\bf order
preserving} (i.e. $\pi(i)<\pi(j)$ whenever $i<j$) bijection $\pi: D(\pi) \rightarrow R(\pi)$ such that $D(\pi),R(\pi) \subset\{0, 1,\ldots, n-1\}$
and for every $i \in D(\pi)$ we have $d(T^ix, T^{\pi(i)}y) < \delta$. Let $|\pi|$ be the cardinality of $D(\pi)$. We set$$\bar{f}_{n,\delta}(x, y) = 1-\frac{\max\{|\pi| : \pi \text{ is an }(n, \delta)\text{-match of }x \text{ and } y\}}{n}$$
and
$$\bar{f}_\delta(x, y) = \limsup_{n\to +\infty}\bar{f}_{n,\delta}(x, y).$$
\begin{de}\label{def:fk-metric}
Define the Feldman-Katok metric of $(X, T)$ by  $$d_{FK_n}(x, y) = \inf\{\delta> 0: \bar{f}_{n,\delta}(x, y) < \delta\}$$
and
$$d_{FK}(x, y) = \inf\{\delta> 0: \bar{f}_\delta(x, y) < \delta\}.$$	
\end{de}

\begin{rem}
It is easy to see that $d_{FK_n}(x, y)$ is  a metric, $d_{FK}(x, y)$ is a pseudometric and $$\limsup_{n\to +\infty}d_{FK_n}(x, y)=d_{FK}(x, y).$$	
\end{rem}

\medskip

For $x,y\in X$ and $n\in\N$, let
$$d_n(x,y)=\max_{0\leq i\leq n-1}d(T^ix,T^iy)$$
and
$$\bar{d}_n(x,y)=\frac{1}{n}\sum_{i=0}^{n-1}d(T^ix,T^iy)$$
be the Bowen metric and the mean metric respectively. It is easy to see  the following lemma:
\begin{lem}\label{lem:relationOfmetrics}
	$d_{FK_n}(x,y)\leq(\bar{d}_n(x,y))^{\frac{1}{2}}, \bar{d}_n(x,y)\leq d_n(x,y).$
\end{lem}
\begin{proof}
Let $\bar{d}_n(x,y)=\delta^2.$ Then
	$$\frac{1}{n}\sum_{0\leq i\leq n-1:d(T^ix,T^iy)\geq \delta}d(T^ix,T^iy)\leq\frac{1}{n}\sum_{i=0}^{n-1}d(T^ix,T^iy)=\delta^2.$$
	Hence $$|\{0\leq i\leq n-1:d(T^ix,T^iy)\geq \delta\}|\leq n\delta.$$
	Denote
$$
D=\{0\leq i\leq n-1:d(T^ix,T^iy)<\delta\}
$$
   and let $\pi=\id$ on $D$.
	Then $\pi$ is an $(n, \delta)$-match of $x,y$ and $1-\frac{|\pi|}{n}\leq \delta.$ Hence $d_{FK_n}(x,y)\leq \delta$.
	
	The second inequality is directly from definitions.
	\end{proof}

\subsection{Topological entropy formulas for  Feldman-Katok metric}

We start with the topological case.  Following the idea of Bowen \cite{bowen},
we first give the analogous notions of spanning  and separated set for $d_{FK_n}$.
\begin{de}
	Let $(X,T)$ be a TDS, $n\in\N$ and $\ep>0$.
	A subset $E$ is said to be $FK$-$(n,\ep)$ {\it spanning} if $\forall x\in X, \exists\ y\in E$ with $d_{FK_n}(x,y)\leq \ep.$ Let $sp_{FK}(n,\ep)$ denote the smallest cardinality of any $FK$-$(n,\ep)$ spanning set.
\end{de}
\begin{de}
	Let $(X,T)$ be a TDS,  $n\in\N$ and $\ep>0$.
	A subset $F$ is said to be {\it  $FK$-$(n,\ep)$ separated} if $x\neq y\in F\Rightarrow d_{FK_n}(x,y)>\ep.$ Let $sr_{FK}(n,\ep)$ denote the largest cardinality of any $FK$-$(n,\ep)$ separated set.
\end{de}
\begin{rem}
	Clearly $sp_{FK}(n,\ep)<\infty$ because of the compactness of $X$.
Similarly we have $sp_{FK}(n,\ep)\leq sr_{FK}(n,\ep)\leq sp_{FK}(n,\frac{\ep}{2})$ and hence $sr_{FK}(n,\ep)<\infty.$
\end{rem}
Now we can prove the formula for topological entropy.
\begin{thm}\label{thm:bowen-FK}
	Let $(X,T)$ be a TDS. Then
	$$h_{top}(X,T)=\lim_{\ep\to0}\liminf_{n\to\infty}\frac{1}{n}\log(sp_{FK}(n,\ep))=\lim_{\ep\to0}\limsup_{n\to\infty}\frac{1}{n}\log(sp_{FK}(n,\ep)),$$
	$$h_{top}(X,T)=\lim_{\ep\to0}\liminf_{n\to\infty}\frac{1}{n}\log(sr_{FK}(n,\ep))=\lim_{\ep\to0}\limsup_{n\to\infty}\frac{1}{n}\log(sr_{FK}(n,\ep)).$$
\end{thm}
\begin{proof}
Let   $\mathcal{U}$ be  a finite open cover of $X$. 	 Let the Lebesgue number of $\mathcal{U}$ be    $2\ep_0.$
 For any $0<\ep<\ep_0$ and $n\in\N,$ let $E$ be a $FK$-$(n,\ep)$ spanning set with $|E|=sp_{FK}(n,\ep).$ By the definitions of $FK$-$(n,\ep)$ spanning set and $d_{FK_n}$ we have
	$$X=\bigcup_{x\in E}\bigcup_{k=[(1-\ep)n]}^n\bigcup_{\substack{\pi:\ |\pi|=k,\\\pi \text{\ is order\ preserving}}		
	}\bigcap_{i\in D(\pi)}T^{-i}B_\ep(T^{\pi(i)}x).$$
	It is clear that $B_\ep(T^{\pi(i)}x)$ is contained in some element of $\mathcal{U}.$ Hence $\bigcap_{i\in D(\pi)}T^{-i}B_\ep(T^{\pi(i)}x)$ is contained in some element of $\bigvee_{i\in D(\pi)}T^{-i}\mathcal{U}.$ Note that $$\bigvee_{i=0}^{n-1}T^{-i}\mathcal{U}=(\bigvee_{i\in D(\pi)}T^{-i}\mathcal{U})\vee(\bigvee_{i\notin D(\pi)}T^{-i}\mathcal{U})$$
	and $$|\bigvee_{i\notin D(\pi)}T^{-i}\mathcal{U}|\leq|\mathcal{U}|^{n-|\pi|}.$$ This implies that each element of $\bigvee_{i\in D(\pi)}T^{-i}\mathcal{U}$ can be covered by $|\mathcal{U}|^{n-|\pi|}$ elements of $\bigvee_{i=0}^{n-1}T^{-i}\mathcal{U}.$  Hence $\bigcap_{i\in D(\pi)}T^{-i}B_\ep(T^{\pi(i)}x)$ can be covered by $|\mathcal{U}|^{n-|\pi|}$ elements of $\bigvee_{i=0}^{n-1}T^{-i}\mathcal{U}.$ Since the number of  order preserving $\pi$ with $|\pi|=k$ is $(C_n^k)^2$, it is easy to see that $X$ can be covered by
	$$|E|\sum_{k=[n(1-\ep)]}^n(C_n^k)^2|\mathcal{U}|^{n-k}$$
	elements of $\bigvee_{i=0}^{n-1}T^{-i}\mathcal{U}.$
	Hence
	\begin{equation*}
	N(\bigvee_{i=0}^{n-1}T^{-i}\mathcal{U})\leq|E|\sum_{k=[n(1-\ep)]}^n(C_n^k)^2|\mathcal{U}|^{n-k}\leq sp_{FK}(n,\ep)(n\ep+2)(C_n^{[n\ep]+1})^2|\mathcal{U}|^{n\ep}.
	\end{equation*}
	It follows that
	$$\frac{\log N(\bigvee_{i=0}^{n-1}T^{-i}\mathcal{U})}{n}\leq\frac{\log sp_{FK}(n,\ep)}{n}+\frac{\log (n\ep+2)}{n}+\frac{2\log C_n^{[n\ep]+1}}{n}+\ep\log|\mathcal{U}|.$$
	By Stirling's formula,
	$$\lim_{n\rightarrow +\infty} \frac{1}{n}\log C_n^{[n\ep]+1}=-(1-\ep)\log (1-\ep)-\ep\log \ep.$$
	We have $$h_{top}(T,\mathcal{U})\leq \lim_{\ep\to0}\liminf_{n\to\infty}\frac{1}{n}\log(sp_{FK}(n,\ep)).$$
	It follows that
	$$h_{top}(X,T)\leq \lim_{\ep\to0}\liminf_{n\to\infty}\frac{1}{n}\log(sp_{FK}(n,\ep)).$$
	On the other hand, since $d_{FK_n}(x,y)\leq d_n(x,y),$
		we have $$h_{top}(X,T)\geq \lim_{\ep\to0}\limsup_{n\to\infty}\frac{1}{n}\log(sp_{FK}(n,\ep)).$$
		
		The second formula for $sr_{FK}$ is directly from the fact that $$sp_{FK}(n,\ep)\leq sr_{FK}(n,\ep)\leq sp_{FK}(n,\frac{\ep}{2}).$$
\end{proof}
\begin{rem}\label{rem:brin-katok:1}
\begin{enumerate}
  \item Note that the method of \cite[Lemma 4.1]{GJ15} for mean-metric case can not be applied to Theorem \ref{thm:bowen-FK} for Feldman-Katok metric case. Furthermore, by Lemma \ref{lem:relationOfmetrics} and  Theorem \ref{thm:bowen-FK} we can give an alternative and simpler proof of \cite[Lemma 4.1]{GJ15}.
  \item We believe that Theorem \ref{thm:bowen-FK} can be  generalized to sequence entropy  case.
\end{enumerate}

\end{rem}

\subsection{Measure-theoretic entropy formulas for  Feldman-Katok metric}
Now we consider the measurable case.

\medskip

First we need some preparations.

Recall that the  {\it edit distance} $\bar{f}_n$ is defined by:
$$\bar{f_n}(x_0x_1\ldots x_{n-1},y_0y_1\ldots y_{n-1})=1-\frac{k}{n},$$
where $k$ is the largest integer such that there exist $$0\leq i_1<\ldots<i_k\leq n-1, 0\leq j_1<\ldots<j_k\leq n-1$$ and $x_{i_s}=y_{j_s}$ for $s=1,\ldots,k.$

Let $(X,T)$ be a TDS. We denote the set of all invariant probability  measures  of $(X,T)$ by $M(X,T)$.
Let $\mu\in M(X,T)$ and $\xi=\{A_1,\ldots,A_{|\xi|}\}$ be  a finite partition of $X$. We can identify the elements in $\bigvee_{i=0}^{n-1}T^{-i}\xi$ and $\{1,\ldots,|\xi|\}^n$ by
$$\bigcap_{i=0}^{n-1} T^{-i}A_{t_i}=(t_0,\ldots,t_{n-1}).$$ Hence for $w\in\{1,\ldots,|\xi|\}^n$ we can talk about $\mu(w)$ and for $A, B\in \bigvee_{i=0}^{n-1}T^{-i}\xi$ we can talk about $\bar{f}_n(A,B).$
For $x\in X,$    the element of  $\xi$ containing $x$ is denoted by $\xi(x)$. We simply write $\xi^n$ instead of $\bigvee_{i=0}^{n-1}T^{-i}\xi$ and $\xi^n(x)$ instead of $(\bigvee_{i=0}^{n-1}T^{-i}\xi)(x).$

\medskip

In 1983, Brin and Katok \cite{bk} introduced the Brin-Katok formula: If $(X,T)$ is a TDS and $\mu\in M(X,T)$, then for $\mu-a.e.\ x\in X$ we have
$$h_{\mu}(T,x):=\lim_{\delta\to0}\liminf_{n\to\infty}-\frac{\log\mu(B_n(x,\delta))}{n}=\lim_{\delta\to0}\limsup_{n\to\infty}-\frac{\log\mu(B_n(x,\delta))}{n},$$
where $h_{\mu}(T,x)$ is $T$-invariant and $\int h_{\mu}(T,x)dx=h_\mu(T).$ Recall that  $B_n(x,\delta)=\{y\in X: d_n(x,y)<\delta\}$ is  the Bowen ball.

\medskip

We replace  Bowen metric by  Feldman-Katok metric and denote
$$B_{d_{FK_n}}(x,\delta)=\{y\in X: d_{FK_n}(x,y)<\delta\}.$$  We have the following theorem:
\begin{thm}\label{thm:Brin-Katok}
	Let $(X,T)$ be a TDS and $\mu\in M(X,T)$. Then for $\mu-a.e.\ x\in X,$
	$$h_{\mu}(T,x)=\lim_{\delta\to0}\liminf_{n\to\infty}-\frac{\log\mu(B_{d_{FK_n}}(x,\delta))}{n}=\lim_{\delta\to0}\limsup_{n\to\infty}-\frac{\log\mu(B_{d_{FK_n}}(x,\delta))}{n}.$$
	
\end{thm}
\begin{proof}	
	Let $\{\mathcal{P}_i\}_{i=1}^\infty$
be a family of increasing finite Borel partitions of $X$ with
 $\diam(\mathcal{P}_i)\to 0$ as $i\to \infty$ and $\mu(\partial\mathcal{P}_i)=0$ for $i\in\N$, here $\diam(\mathcal{P})=\max\limits_{A\in \mathcal{P}}\diam(A)$ and $\partial\mathcal{P}=\bigcup\limits_{A\in\mathcal{P}}\partial A.$ Let $$X_0=\{x\in X:h_\mu(T,x)<+\infty\},\ X_\infty=\{x\in X:h_\mu(T,x)=+\infty\}.$$ Without loss of generality, we assume $\mu(X_0)>0$ and $\mu(X_\infty)>0.$

We first show that:
$$\lim_{\delta\to0}\liminf_{n\to\infty}-\frac{\log\mu(B_{d_{FK_n}}(x,\delta))}{n}\geq h_{\mu}(T,x), \ a.e.\ x\in X_0.$$

Fix $\ep>0.$ Since $$h_\mu(\mathcal{P}_i,x)\to h_\mu(T,x), \ a.e.  \ x\in X_0,$$ we can find some $\mathcal{P}\in\{\mathcal{P}_i\}_{i=1}^\infty$ such that $\mu(A)>\mu(X_0)-\frac{\ep}{2},$
where $$A=\{x\in X_0:|h_\mu(\mathcal{P},x)-h_\mu(T,x)|<\frac{\ep}{2}\}\subset X_0.$$
For $\tau>0$, let
$$U_\tau(\mathcal{P})=\{x\in X:\text{the ball }B(x,\tau) \text{ is not contained in }\mathcal{P}(x)\}.$$
Recall $\mathcal{P}(x)$ denotes the element of the partition $\mathcal{P}$ containing $x$. Since $$\bigcap_{\tau>0}U_\tau(\mathcal{P})=\partial\mathcal{P},$$ we have $\mu(U_\tau(\mathcal{P}))\to 0$ as $\tau\to 0.$
Choose some $0<\kappa<\ep$ such that
\begin{equation}\label{2}
2\kappa\log|\mathcal{P}|+4(-\kappa\log\kappa-(1-\kappa)\log(1-\kappa))<\frac{\ep}{2}.
\end{equation}
Then there exists $0<\delta_0<\frac{\kappa}{2}$ such that $\mu(U_{\delta_0}(\mathcal{P}))<(\frac{\kappa}{4})^2.$
We can choose some $L\in\N$ such that $\mu(E_L)>1-\frac{\ep}{2},$ where $$E_L=\{x\in X:\frac{1}{n}\sum_{i=0}^{n-1}1_{U_{\delta_0}(\mathcal{P})}(T^ix)\leq \frac{\kappa}{2},\ \forall\ n\geq L\}.$$
Hence $A\cap E_L\subset X_0$ and $\mu(A\cap E_L)>\mu(X_0)-\ep.$
Let $$A_k=\{x\in A\cap E_L: k\ep\leq h_\mu(T,x)<(k+1)\ep\}.$$
We have $\bigcup_{k=0}^\infty A_k=A\cap E_L,$ hence there exists some $N$ such that $\mu(\bigcup_{k=0}^N A_k)>\mu(X_0)-\ep.$ By Shannon-McMillan-Breiman theorem, we have
$$-\frac{\log\mu(\mathcal{P}^n(x))}{n}\to h_\mu(\mathcal{P},x),\ a.e. \ x\in X.$$
Hence we can find $B_k\subset A_k, 0\leq k\leq N$ such that $$\mu(\bigcup_{k=0}^N B_k)>\mu(X_0)-\ep,$$
 and
$\exists\ N_1\in\N, \forall\ n>N_1, \forall\ x\in B_k$, we have
\begin{equation}\label{1}
-\frac{\log\mu(\mathcal{P}^n(x))}{n}>h_\mu(\mathcal{P},x)-\frac{\ep}{2}>h_\mu(T,x)-\ep\geq (k-1)\ep.
\end{equation}
The second inequality holds because $x\in A.$

Let $N_0=\max\{N_1,L\}$. Let $n>N_0.$

{\bf Claim:} For $x\in B_k$, we have $$B_{d_{FK_n}}(x,\delta_0)\subset \bigcup_{w:\bar{f}_n(w,\mathcal{P}^n(x))<\kappa}w .$$

Proof of claim: Let $y\in B_{d_{FK_n}}(x,\delta_0).$  By the definition of $d_{FK_n}$,  there exists  an $(n, \delta_0)$-match $\pi$ of $x$ and $y$ with $|\pi|>(1-\delta_0)n$.
Since $x\in E_L,$ we have $$|\{0\leq i\leq n-1: T^ix\notin U_{\delta_0}(\mathcal{P})\}|>(1-\frac{\kappa}{2})n.$$
Note that $$|\pi|>(1-\delta_0)n>(1-\frac{\kappa}{2})n.$$
We have $$|\{i\in D(\pi): T^ix\notin U_{\delta_0}(\mathcal{P})\}|>(1-\kappa)n.$$
Let the set on the leftside be $D^\prime$. Hence $D^\prime\subset D(\pi)$ and $|D^\prime|>(1-\kappa)n.$
 For $i\in D^\prime$, since $T^{i}x\notin U_{\delta_0}(\mathcal{P})$, we have $ B(T^{i}x,\delta_0)\subset \mathcal{P}(T^{i}x).$ Note that $d(T^{i}x,T^{\pi(i)}y)<\delta_0$ by the definition of $\pi$. Hence $T^{\pi(i)}y\in\mathcal{P}(T^{i}x)$ . It follows that
 $$ \bar{f}_n(\mathcal{P}^n(x),\mathcal{P}^n(y))<\kappa$$ by the definition of $\bar{f}_n$.
The claim is proved.

\medskip

Now we estimate the measure of the set $$\{x\in B_k: \mu( B_{d_{FK_n}}(x,\delta_0))>e^{-n(k-2)\ep}\}.$$
Note that
\begin{equation*}
\begin{split}
&\{x\in B_k: \mu( B_{d_{FK_n}}(x,\delta_0))>e^{-n(k-2)\ep}\}\\
&\subset \{x\in B_k: \mu(\bigcup_{w:\bar{f}_n(w,\mathcal{P}^n(x))<\kappa}w)>e^{-n(k-2)\ep}\} \\
&\subset \{x\in B_k: \exists\ w\in B_{\bar{f}_n}(\mathcal{P}^n(x),\kappa)\ s.t.\ \mu(w)>\frac{e^{-n(k-2)\ep}}{ (C^{[n\kappa]}_n)^2|\mathcal{P}|^{n\kappa}}\}\\
&\subset\bigcup_{w\cap B_k\neq\emptyset}\{w: \exists\ w^\prime\in  B_{\bar{f}_n}(w,\kappa)\ s.t.\ \mu(w^\prime)>\frac{e^{-n(k-2)\ep}}{ (C^{[n\kappa]}_n)^2|\mathcal{P}|^{n\kappa}}\}.
\end{split}
\end{equation*}
The first relation is from claim and the second relation holds because the number of elements in $B_{\bar{f}_n}(\mathcal{P}^n(x),\kappa)$ is not more than $(C^{[n\kappa]}_n)^2|\mathcal{P}|^{n\kappa}.$ We only need to estimate the measure of the last set. The number of $w$ with $w\cap B_k\neq \emptyset$ is not more than
$$e^{n(k-2)\ep}((C^{[n\kappa]}_n)^2|\mathcal{P}|^{n\kappa})^2.$$
We estimate the measure of a single $w.$ Let  $x\in w\cap B_k.$ Hence $\mathcal{P}^n(x)=w.$ Since $x\in B_k,$ by (\ref{1}) we have
$$-\frac{\log\mu(\mathcal{P}^n(x))}{n}>(k-1)\ep.$$ That is $\mu(w)<e^{-n(k-1)\ep}$.
Hence we have
\begin{equation*}
\begin{split}
&\mu(\{x\in B_k: \mu( B_{d_{FK_n}}(x,\delta_0))>e^{-n(k-2)\ep}\})\\
&<e^{-n(k-1)\ep}e^{n(k-2)\ep}((C^{[n\kappa]}_n)^2|\mathcal{P}|^{n\kappa})^2\\
&=e^{-n\ep}((C^{[n\kappa]}_n)^2|\mathcal{P}|^{n\kappa})^2.
\end{split}
\end{equation*}
By Stirling's formula $C^{[n\kappa]}_n\approx e^{n(-\kappa\log\kappa-(1-\kappa)\log(1-\kappa))}.$ Hence $$(C^{[n\kappa]}_n)^4|\mathcal{P}|^{2n\kappa}\approx e^{4n(-\kappa\log\kappa-(1-\kappa)\log(1-\kappa))+2n\kappa\log|\mathcal{P}|}<e^\frac{n\ep}{2}$$
by (\ref{2}). Hence $$\sum e^{-n\ep}((C^{[n\kappa]}_n)^2|\mathcal{P}|^{n\kappa})^2$$ is convergent. By the Borel-Cantelli lemma, we have
$$\liminf_{n\to\infty}-\frac{\log\mu(B_{d_{FK_n}}(x,\delta_0))}{n}\geq (k-2)\ep>h_\mu(T,x)-3\ep,\ a.e.\ x\in B_k.$$
The last inequality holds because $x\in A_k.$ Hence  we have
$$\lim_{\delta\to0}\liminf_{n\to\infty}-\frac{\log\mu(B_{d_{FK_n}}(x,\delta))}{n}>h_\mu(T,x)-3\ep, \ a.e. \ x\in \bigcup_{k=0}^N B_k.$$
Now we get that for any $\ep>0,$ there is some subset $E$ of $X_0$ with $\mu(E)>\mu(X_0)-\ep$ such that $$\lim_{\delta\to0}\liminf_{n\to\infty}-\frac{\log\mu(B_{d_{FK_n}}(x,\delta))}{n}>h_\mu(T,x)-3\ep,\ a.e.\ x\in E.$$ Hence
$$\lim_{\delta\to0}\liminf_{n\to\infty}-\frac{\log\mu(B_{d_{FK_n}}(x,\delta))}{n}\geq h_{\mu}(T,x),\ a.e.\ x\in X_0.$$

\medskip

Similarly we can prove $$\lim_{\delta\to0}\liminf_{n\to\infty}-\frac{\log\mu(B_{d_{FK_n}}(x,\delta))}{n}=+\infty,\ a.e.\ x\in X_\infty.$$

Combining the fact that $d_{FK_n}\leq d_n$, we have$$\lim_{\delta\to0}\limsup_{n\to\infty}-\frac{\log\mu(B_{d_{FK_n}}(x,\delta))}{n}\leq h_{\mu}(T,x),\ a.e.\ x\in X. $$
	\end{proof}
\begin{rem} \label{rem:brin-katok:2}
\begin{enumerate}
  \item It is fair to mention  that the proof of Theorem \ref{thm:Brin-Katok} follows the original spirit of  Brin-Katok formulas \cite{bk}, nonetheless, there are more parameters appearing in the Feldman-Katok metric case and some finer estimates  and  techniques  need to be developed.

  \item
	By the definition of $d_{FK_n}$, we know that $d_{FK_n}(x, Tx)\leq \frac{1}{n}$. Hence $$B_{d_{FK_n}}(x,\delta)\subset B_{d_{FK_n}}(Tx,\delta+\frac{1}{n})$$ and $$B_{d_{FK_n}}(Tx,\delta)\subset B_{d_{FK_n}}(x,\delta+\frac{1}{n}).$$  Then we can deduce directly (does not rely on the continuity of $T$) that the rightside of the formula in Theorem \ref{thm:Brin-Katok} is $T$-invariant.
\end{enumerate}
\end{rem}

\medskip

We similarly introduce the measure-theoretic version of $sp_{FK}(n,\ep)$ as follows.
\begin{de}
	Let $(X,T)$ be a TDS, $n\in\N,\ \ep>0$ and $\mu\in M(X,T)$. Denote
	$$sp_{FK}(d,\mu,n,\ep)=\min\{m\in \N:\exists\ x_1,x_2,\ldots,x_m\in X\ s.t.\ \mu(\bigcup_{i=1}^mB_{d_{FK_n}}(x_i,\ep))>1-\ep\}.$$
		We often use $sp_{FK}(\mu,n,\ep)$ for simplicity when the metric $d$ is clear from the context.
\end{de}
Now  we   give  Katok's formula for  Feldman-Katok metric.
Here we  follow the idea of \cite[Theorem 2.4]{adv} with suitable modifications.

\begin{thm} \label{thm:Katok-formula}
	Let $(X,T)$ be a TDS and $\mu\in M(X,T)$. Then we have
	$$h_\mu(T)\le \lim_{\epsilon\rightarrow 0}\liminf_{n\rightarrow \infty}\frac{1}{n}\log sp_{FK}(\mu,n, \epsilon).$$
	If $\mu$ is ergodic, then
	$$ h_\mu(T)=\lim_{\epsilon\rightarrow 0}\limsup_{n\rightarrow \infty}\frac{1}{n}\log sp_{FK}(\mu,n, \epsilon)=\lim_{\epsilon\rightarrow 0}\liminf_{n\rightarrow \infty}\frac{1}{n}\log sp_{FK}(\mu,n, \epsilon).$$
	
\end{thm}
\begin{proof}
	Note that $d_{FK_n}\leq d_n,$ we only need to prove that
	\begin{equation*}
	h_\mu(T)\le \lim_{\epsilon\rightarrow 0}\liminf_{n\rightarrow \infty}\frac{1}{n}\log sp_{FK}(\mu,n, \epsilon).
	\end{equation*}
	Given  a Borel partition $\eta$ of $X$ and $\delta>0$, we need to show
	$$h_\mu(T,\eta)< \lim_{\epsilon\rightarrow 0}\liminf_{n\rightarrow \infty}\frac{1}{n}\log sp_{FK}(\mu,n, \epsilon)+2\delta.$$

	Let $|\eta|=k.$
	Take $0<\kappa<\frac{1}{2}$ with $$-4\kappa\log 2\kappa-2(1-2\kappa)\log (1-2\kappa)+4\kappa \log (k+1) <\delta.$$
	It is easy to see that there is a Borel partition $\xi=\{B_1, \ldots, B_k,B_{k+1}\}$ of $X$ such that
	$B_i$ is closed for $1\le i\le k$, $\mu(B_{k+1})<\kappa^2$, and
	\begin{equation*}
	h_\mu(T,\eta)\le h_\mu(T,\xi)+\delta.
	\end{equation*}
	Let $K=\cup_{i=1}^kB_i$ and $b=\min_{1\le i<j\le k} d(B_i,B_j)$. Then $\mu(K)>1-\kappa^2$ and $b>0$.
	
	\medskip
	
	Let $0<\epsilon<\frac{b}{2}$, $n\in\mathbb{N}$. By the definition of $sp_{FK}(\mu,n,\ep)$, there are $x_1,\ldots, x_{m(n)}\in X$ such that
	$$\mu\big(\bigcup_{i=1}^{m(n)} B_{d_{FK_n}}(x_i,\epsilon)\big)>1-\epsilon,$$ where
	$m(n)=sp_{FK}(\mu, n,\epsilon).$ Let $$F_n=\bigcup_{i=1}^{m(n)} B_{d_{FK_n}}(x_i,\epsilon).$$

	Let $$ E_n=\{x\in X:\frac{|\{0\leq i\leq n-1: T^ix\in K\}|}{n}\le 1-\kappa\}.$$
	
	It is easy to see that
	$\mu(E_n)<\kappa$. Put $W_n=(K\cap F_n)-E_n$. Then
	$$\mu(W_n)>1-\kappa^2-\kappa-\epsilon>1-2\kappa-\epsilon.$$
	For $z\in W_n$, we have
	\begin{equation*}
	\frac{|\{0\leq i\leq n-1: T^iz\in K\}|}{n}>1-\kappa.
	\end{equation*}

	{\bf Claim}:
	\begin{equation*}
	|\{A\in \bigvee_{j=0}^{n-1}T^{-j}\xi: A\cap B_{d_{FK_n}}(x_i,\epsilon)\cap W_n \not=\emptyset\}|\le (C_n^{[(1-2\kappa-2\ep)n]})^2 \cdot (k+1)^{[(2\kappa+2\ep)n]+1}.
	\end{equation*}
	
	Proof of  claim: Let
	$$x\in A_1\cap B_{d_{FK_n}}(x_i,\epsilon)\cap W_n, \ y\in A_2\cap B_{d_{FK_n}}(x_i,\epsilon)\cap W_n,$$ where $A_1,A_2\in \bigvee_{j=0}^{n-1}T^{-j}\xi$. Then $d_{FK_n}(x,y)<2\epsilon$.
	Denote $$D_x=\{0\leq j\leq n-1: T^jx\in K\},\ D_y=\{0\leq j\leq n-1: T^jy\in K\}.$$ Then $$|D_x|> (1-\kappa)n,\ |D_y|> (1-\kappa) n.$$
	Since $d_{FK_n}(x,y)<2\epsilon$,  let $\pi$ be an $(n, 2\ep)$-match  of $x \text{ and } y$ with $|\pi|>(1-2\ep)n$.
	Note that $$|\pi^{-1}\left (\pi (D(\pi)\cap D_x)\cap D_y\right )|>(1-2\kappa-2\ep)n.$$
Let the set on the leftside be $D^\prime.$
Then $|D^\prime|>(1-2\kappa-2\ep)n$. For every $j\in D^\prime,$
	we have
	$$d(T^{j}x, T^{\pi(j)}y)<2\ep<b,$$ and $\ T^{j}x,T^{\pi(j)}y\in K.$
	Hence $T^{j}x, T^{\pi(j)}y$ must be in  one of the  same element of $\{B_1, \ldots, B_k\}.$
	It follows that $$\bar{f}_n(A_1,A_2)\leq 2\kappa+2\ep.$$
	Note that
	the number of  $A$  satisfying $$\bar{f}_n(A_1,A)\leq2\kappa+2\ep$$ is not more than $$(C_n^{[(1-2\kappa-2\ep)n]})^2 \cdot (k+1)^{[(2\kappa+2\ep)n]+1}.$$
	The proof of the claim is finished.
	
	\medskip
	
	Now we estimate $h_\mu(T,\xi)$. We have
	\begin{align*}
	H_\mu(\bigvee_{j=0}^{n-1}T^{-j}\xi)&\le
	H_\mu(\bigvee_{j=0}^{n-1}T^{-j}\xi\vee \{ W_n,X\setminus W_n\})\\
	&\le \log \big( \sum_{i=1}^{m(n)} |\{A\in\bigvee_{j=0}^{n-1}T^{-j}\xi: A\cap B_{d_{FK_n}}(x_i,\epsilon)\cap W_n \not=\emptyset\}|\big)\\
	&\hskip0.8cm +(2\kappa+\epsilon)n \log (k+1)+\log 2.
	\end{align*}
	Thus
	\begin{align*}
	h_\mu(T,\xi)&=\lim_{n\rightarrow +\infty} \frac{1}{n}H_\mu(\bigvee_{j=0}^{n-1}T^{-j}\xi)\\
	& \leq \lim_{\epsilon\rightarrow 0}\liminf_{n\rightarrow \infty}\frac{1}{n}\log sp_{FK}(\mu,n, \epsilon)+2\limsup_{n\rightarrow +\infty} \frac{1}{n}\log C_n^{[(1-2\kappa-2\ep)n]} +(4\kappa+3\epsilon) \log (k+1).
	\end{align*}
	By Stirling's formula and let $\epsilon\rightarrow 0$  we obtain
	\begin{equation*}
	\begin{split}
	h_\mu(T,\xi)&\leq \lim_{\epsilon\rightarrow 0}\liminf_{n\rightarrow \infty}\frac{1}{n}\log sp_{FK}(\mu,n, \epsilon)-4\kappa\log 2\kappa-2(1-2\kappa)\log (1-2\kappa)+4\kappa \log (k+1)\\ &<\lim_{\epsilon\rightarrow 0}\liminf_{n\rightarrow \infty}\frac{1}{n}\log sp_{FK}(\mu,n, \epsilon)+\delta.
	\end{split}
	\end{equation*}
	Hence
	$$h_\mu(T,\eta)< h_\mu(T,\xi)+\delta<\lim_{\epsilon\rightarrow 0}\liminf_{n\rightarrow \infty}\frac{1}{n}\log sp_{FK}(\mu,n, \epsilon)+2\delta.$$
\end{proof}
\begin{rem} \label{rem:brin-katok:3}
	\begin{enumerate}
		\item \label{rem:brin-katok:3:1} 	Theorem \ref{thm:bowen-FK} can also be obtained by Theorem \ref{thm:Katok-formula} and the variational principle.
		\item\label{rem:brin-katok:3:2}
		 We believe that Theorems \ref{thm:Brin-Katok} and \ref{thm:Katok-formula} can also be  generalized   to  the conditional  case (see \cite{zxm,chen}).
\item \label{rem:brin-katok:3:3} From Theorem \ref{thm:Katok-formula} it is easy to see that if $\{sp_{FK}(\mu,n,\ep)\}_{n=1}^\infty$ is bounded for every $\ep>0$, then the system  must have zero entropy.
	\end{enumerate}
\end{rem}

\medskip

When studying the Sarnak conjecture, the authors in \cite{adv} considered the  measure complexity of a system in the mean metric. The advantage
to use mean metric is that it is an isomorphism invariant (see Proposition 2.2 therein for details). We explain this is still valid for  Feldman-Katok metric.

Let $U(n):\mathbb{N}\rightarrow [1,+\infty)$ be an increasing sequence with $\lim_{n\rightarrow +\infty} U(n)=+\infty$.
We say that the {\it measure FK-complexity of $(X,d,T,\mu)$ is weaker than $U(n)$}
if $$\liminf_{n\rightarrow +\infty} \frac{sp_{FK}(d,\mu,n,\ep)}{U(n)}=0$$ for any $\epsilon>0$. Following the idea of \cite[Proposition 2.2]{adv}, we can prove the following proposition. The  proof is moved to Appendix.

\begin{prop}\label{prop:measure-complexity}
	Assume $(X,\mathcal{B}(X),T,d,\mu)$ is measurably isomorphic to $(Y,\mathcal{B}(Y),S,d',\nu)$,
	and $U(n):\mathbb{N}\rightarrow [1,+\infty)$ is an increasing sequence with $\lim_{n\rightarrow +\infty} U(n)=+\infty$.
	Then the measure FK-complexity of $(X,d,T,\mu)$  is
	weaker than $U(n)$ if and only if the measure FK-complexity of $(Y,d',S,\nu)$ is weaker than $U(n)$.
\end{prop}

\section{the weak-mean pseudometric}\label{sect:weak-mean pseudometric}
By Lemma \ref{lem:relationOfmetrics} we know that ${d}_{FK_n}(x, y)\leq(\bar{d}_n(x,y))^{\frac{1}{2}}\leq d_n(x,y)^{\frac{1}{2}}.$ It is natural to ask that does there exist a weaker metric (in the mean form) so that the entropy formulas are still valid? Is the order-preservation of $\pi$ in the definition of Feldman-Katok metric necessary?

In this section we study another pseudometric related to Feldman-Katok pseudometric and discuss the questions above.

\medskip

If we ignore the order in the definition of Feldman-Katok metric, we have an analogous definition:

For $x, y\in X, \delta>0$ and $n\in\N$, we define an $(n, \delta)^*$-match of $x$ and $y$ to be a  bijection $\pi$ (not necessary order-preserving)$: D(\pi) \rightarrow R(\pi)$ such that $D(\pi),R(\pi) \subset\{0, 1,\ldots, n-1\}$
and for every $i \in D(\pi)$ we have $d(T^ix, T^{\pi(i)}y) < \delta$. The fit $|\pi|$ of an $(n, \delta)^*$-match $\pi$
is the cardinality of $D(\pi)$. We set$$\tilde{f}_{n,\delta}(x, y) = 1-\frac{\max\{|\pi| : \pi \text{ is an }(n, \delta)^*\text{-match of }x \text{ and } y\}}{n}$$
and
$\tilde{f}_\delta(x, y) = \limsup_{n\to +\infty}\tilde{f}_{n,\delta}(x, y)$.
$$\tilde{d}_{FK_n}(x, y) = \inf\{\delta> 0: \tilde{f}_{n,\delta}(x, y) < \delta\},$$
$$\tilde{d}_{FK}(x, y) = \inf\{\delta> 0: \tilde{f}_\delta(x, y) < \delta\}.$$

\begin{rem}
 $\tilde{d}_{FK}$ is indeed a  pseudometric. By definition we have
 $\tilde{d}_{FK_n}(x, y)\leq {d}_{FK_n}(x, y),$ $\tilde{d}_{FK}(x, y)\leq {d}_{FK}(x, y).$
\end{rem}

We explain that the entropy formula for $\tilde{d}_{FK}$ is not valid.
In fact, as we observed,  $\tilde{d}_{FK}$ is just the  pseudometric defined in \cite{zhen}, which is  called  {\it weak-mean pseudometric}.

\medskip

The weak-mean pseudometric is defined as follows:

For $x,y\in X$ and $n\in\N,$ define
$$F_n(x,y)=\frac{1}{n}\inf_{\sigma\in S_n}\sum_{k=0}^{n-1}d(T^kx,T^{\sigma(k)}y)$$
and
	$$F(x,y)=\limsup_{n\to \infty}\frac{1}{n}\inf_{\sigma\in S_n}\sum_{k=0}^{n-1}d(T^kx,T^{\sigma(k)}y).$$
Where $S_n$ denotes the set of all  permutations of $\{0,1\ldots,n-1\}$.

For the properties of weak-mean pseudometric, please see \cite{zhen} for details.

\medskip

We have the following observation:
\begin{thm}	\label{thm:FK=F}
Let $(X,T)$ be a TDS, then	$\tilde{d}_{FK}$ and $F$ are equivalent.
\end{thm}
\begin{proof}
Without loss of generality, assume $\diam(X)\leq1$.	Let $$a=F(x,y)=\limsup_{n\to \infty}\frac{1}{n}\inf_{\sigma\in S_n}\sum_{k=0}^{n-1}d(T^kx,T^{\sigma(k)}y).$$
	For $\ep>0$, we have $$\frac{1}{n}\inf_{\sigma\in S_n}\sum_{k=0}^{n-1}d(T^kx,T^{\sigma(k)}y)<a+\ep$$ for $n$ sufficiently large. Fix $n$ and choose $\sigma\in S_n$ such that $$\frac{1}{n}\sum_{k=0}^{n-1}d(T^kx,T^{\sigma(k)}y)<a+\ep.$$
	Let $$D=\{0\leq k\leq n-1: d(T^kx,T^{\sigma(k)}y)<(a+\ep)^{\frac{1}{2}}\}.$$
	It is easy to estimate that $|D|> (1-(a+\ep)^{\frac{1}{2}})n$.
	Let $\pi=\sigma$ on $D$.  Hence $\pi$ is an $(n, (a+\ep)^{\frac{1}{2}})^*$-match of $x$ and $y$. It follows that
	$$\tilde{f}_{n,(a+\ep)^{\frac{1}{2}}}(x, y)\leq 1-\frac{|D|}{n}<(a+\ep)^{\frac{1}{2}}.$$  By the definition of  $\tilde{d}_{FK}$, we have  $\tilde{d}_{FK}(x,y)\leq (a+\ep)^{\frac{1}{2}}$. Let $\ep\to 0$, we have $\tilde{d}_{FK}(x,y)\leq a^{\frac{1}{2}}=F(x,y)^{\frac{1}{2}}$.
	
On the other hand,	let $b=\tilde{d}_{FK}(x, y)$. For $\ep>0$, there exists $0<\delta<b+\ep$ such that $\tilde{f}_\delta(x, y) < \delta.$ Hence $\tilde{f}_{n,\delta}(x, y)<\delta$ for $n$ sufficiently large. Fix $n$.  Let  $\pi$ be an $(n, \delta)^*$-match of $x,y$ with $|\pi|>(1-\delta)n$. Choose some $\sigma\in S_n$ such that $\sigma=\pi$ on $D(\pi)$. We have \begin{equation*}
	\begin{split}
	&\frac{1}{n}\inf_{\sigma\in S_n}\sum_{k=0}^{n-1}d(T^kx,T^{\sigma(k)}y)\leq\frac{1}{n}\sum_{k=0}^{n-1}d(T^kx,T^{\sigma(k)}y)\\&=\frac{1}{n}\sum_{k\in D(\pi)}d(T^kx,T^{\sigma(k)}y)+\frac{1}{n}\sum_{k\notin D(\pi)}d(T^kx,T^{\sigma(k)}y)\\
	&<2\delta<2(b+\ep).
	\end{split}
	\end{equation*}
	Hence $$F(x,y)=\limsup_{n\to \infty}\frac{1}{n}\inf_{\sigma\in S_n}\sum_{k=0}^{n-1}d(T^kx,T^{\sigma(k)}y)\leq2(b+\ep).$$
	Let $\ep\to0$, we have $F(x,y)\leq2b=2\tilde{d}_{FK}(x, y).$
\end{proof}

Now we will discuss $F$ instead of $\tilde{d}_{FK}.$

\medskip

The notion of $F$-equicontinuity  was introduced in \cite{zhen}.
\begin{de}
		Let $(X, T)$ be a TDS. We call $(X,T)$ is {\it$F$-equicontinuous} if
		for every $\epsilon>0$, there exists $\delta>0$ such that
		for every $x,y\in X$, $d(x, y) < \delta\Rightarrow F(x, y) <\epsilon$.
\end{de}
\begin{rem}
	\begin{enumerate}
		\item 	In \cite{c}, the authors proved that $F$-equicontinuity is equivalent to a stronger statement, called $\{F_n\}$-equicontinuity: For every $\epsilon>0$, there exists $\delta>0$ such that
		for every $x,y\in X$, $d(x, y) < \delta\Rightarrow F_n(x, y) <\epsilon, \forall \ n\in \N.$
		\item
			In \cite{zhen}, the authors proved that a TDS is uniquely ergodic if and only if $F(x,y)= 0$ for all $x,y\in X$. In particular, it is $\{F_n\}$-equicontinuous. In this case, if we analogously define the $F$-$(n,\ep)$ spanning set and $sp_{F}(n,\ep)$, then  $\{F_n\}$-equicontinuity implies that $\{sp_{F}(n,\ep)\}_{n=1}^\infty$ is bounded for every $\ep>0$. While there exists  a uniquely ergodic TDS with positive entropy. In fact, by Jewett-Krieger theorem (please see \cite{bookg} etc.) every ergodic system has a uniquely ergodic topological model. Hence by
			the variational principle,  the uniquely ergodic topological model of an ergodic positive entropy system has positive entropy.
	\end{enumerate}
\end{rem}

\medskip

In \cite{loose}, the authors used the notion of {\it $\mu$-$FK$-equicontinuity} to characterize loosely Kronecker system.
Here we can similarly define the notion of $\mu$-$F$-equicontinuity:
\begin{de}
	Let $(X, T)$ be a TDS and $\mu\in M(X,T)$. We
	say that $(X, T )$ is {\it $\mu$-$F$-equicontinuous} if for every $\tau>0$ there exists a compact set
	$M\subset X$ with $\mu(M)>1-\tau$, such that for every $\epsilon>0$, there exists $\delta>0$ such that
	for every $x,y\in M$, $d(x, y) < \delta\Rightarrow F(x, y) <\epsilon$.
\end{de}
\begin{rem}
	We can prove $\mu$-$F$-equicontinuity is equivalent to a stronger statement, we call it $\mu$-$\{F_n\}$-equicontinuity: For every $\tau>0$ there exists a compact set
	$M\subset X$ with $\mu(M)>1-\tau$, such that for every $\epsilon>0$, there exists $\delta>0$ such that
	for every $x,y\in M$, $d(x, y) < \delta\Rightarrow F_n(x, y) <\epsilon,\ \forall\ n\in\N$. Please see Appendix.
\end{rem}
We explain that the order in the
definition of Feldman-Katok metric is crucial in obtaining an entropy formula.
If we ignore the order, things may become quite different, and to obtain an entropy formula in this case is hopeless.

\medskip

First we need some preparations.

Let $(X,T)$ be a TDS and $x\in X.$ We call $\mu\in M(X,T)$ is a {\em distribution measure} of  $x$ if $\mu$ is a limit of some subsequence of $\{\frac{1}{n}\sum_{i=0}^{n-1}\delta_{T^ix}\}_{n=1}^{\infty}$. The set of all distribution measures of $x$ is denoted by $w(x)$. Let $2^X=\{A\subset X: A \text{\ is\ a\ non-empty\ closed\ subset\ of\ }X\}$. The Hausdorff metric on $2^X$ is defined as follows: for every $A, B\in 2^X,$ we set $$d_H(A,B)=\max\{\inf\{\ep>0:B\subset B_\ep(A)\}, \inf\{\ep>0:A\subset B_\ep(B)\}\},$$
where $B_\ep(A)=\{x\in X: d(x,A)<\ep\}.$ Let $H$ be the Hausdorff metric on $2^{M(X,T)}$.

 The following theorem is helpful:
\begin{thm}\label{thm2}\cite[Corollary 3.8.]{c}
	Let $(X,T)$ be a TDS and $x\in X.$
	If $w(x)$ consists of a single point, then 	for every $\ep>0$, there is $\delta > 0$ such that for every $y\in X$ with
	$H(w(x),w(y))<\delta$, we have $F(x,y)<\ep. $
\end{thm}
We will see that the notion of $\mu$-$F$-equicontinuity is quite ``weak". In fact, we have the following theorem:
\begin{thm}\label{thm:F-eq-weak}
Let $(X,T)$ be a TDS and $\mu\in M(X,T)$, then	$(X, T)$ is $\mu$-$F$-equicontinuous.
\end{thm}
\begin{proof}
	Let $$X_0=\{x\in X:f^*(x):=\lim_{n\to+\infty}\frac{1}{n}\sum_{i=0}^{n-1}f(T^ix)\text{ exists,\ }\forall\ f\in C(X)\}.$$
	By Birkhoff ergodic theorem, we have $\mu(X_0)=1$.
	Let $\{f_n\}_{n=1}^\infty$ be a dense subset of $C(X).$
	
	{\bf Claim:} $\forall\ \epsilon,\tau>0,$ there exist  compact set $K\subset X$ with $\mu(K)>1-\tau$ and $\delta>0$ such that $x,y\in K$ and $d(x,y)<\delta\Rightarrow H(w(x),w(y))<\epsilon.$
	
	Proof of  claim: $\forall\ \epsilon,\tau>0,$
	choose $N\in\N$ such that $\frac{1}{2^N}<\frac{\ep}{2}$. Since $f^*_1,\ldots,f^*_N$ are measurable, there exists $K\subset X$ with $\mu(K)>1-\tau$ such that $f^*_1,\ldots,f^*_N$ are equicontinuous on $K$.
	Since $\mu(X_0)=1$, we can assume $K\subset X_0$.
	Hence there exists $\delta>0$ such that for $x,y\in K$ and $d(x,y)<\delta$, we have $$|f^*_i(x)-f^*_i(y)|<\frac{\ep||f_i||_\infty}{2N}, 1\leq i\leq N.$$
	Since $x,y\in X_0,$ we can assume $w(x)=\{\upsilon\}, w(y)=\{\lambda\}$. Hence \begin{equation*}
	\begin{split}
	H(w(x),w(y))=d(\upsilon,\lambda)&=\sum_{n=1}^\infty\frac{|\int f_nd\upsilon-\int f_nd\lambda|}{2^{n+1}||f_n||_\infty}=\sum_{n=1}^\infty\frac{| f_n^*(x)- f_n^*(y)|}{2^{n+1}||f_n||_\infty}\\
	&=\sum_{n=1}^N\frac{| f_n^*(x)- f_n^*(y)|}{2^{n+1}||f_n||_\infty}+\sum_{n=N+1}^\infty\frac{| f_n^*(x)- f_n^*(y)|}{2^{n+1}||f_n||_\infty}\\
	&<\frac{\ep}{2}+\frac{\ep}{2}=\ep.
	\end{split}
	\end{equation*}
	The proof of claim is finished.
	
	\medskip
	
	Given $\tau > 0$. For any $l\in \N$, by  claim there exist $K_l$ with $\mu(K_l)>1-\frac{\tau}{2^l}$ and $\delta_l > 0$ such that
	$$x, y\in K_l,\ d(x, y)<\delta_l\Rightarrow H(w(x),w(y))< \frac{\tau}{2^l}.$$
	Let $K=\bigcap_{l=1}^\infty K_l$. Then $\mu(K)>1-\tau$. Without loss of generality, we assume $K$ is compact and $K\subset X_0$. We next show that $F$ is equicontinuous on $K$.
	
	Let $\ep>0$. For any $x\in K$, since $x\in X_0$, by
	Theorem \ref{thm2}, there exists $\eta_x>0$ such that for $y\in X,$ $$H(w(x),w(y))<\eta_x\Rightarrow F(x,y)<\frac{\ep}{2}.$$ Choose $\l_x\in\N$ with $
	\frac{\tau}{2^{l_x}}<\eta_x$. Now we have $$K\subset \bigcup_{x\in K}B(x,\delta_{l_x}).$$ Since $K$ is compact, let $\delta>0$ be the Lebesgue number  of the cover above. For any $x,y\in K$ with $d(x,y)<\delta,$ there must be some $z\in K$ such that $x,y \in B(z,\delta_{l_z})$. Since  $x,z\in K\subset K_{l_z}$ and $d(x,z)<\delta_{l_z}$,  we have $$H(w(x),w(z))< \frac{\tau}{2^{l_z}}<\eta_z.$$
	It follows that $F(x,z)<\frac{\ep}{2}.$ Similarly $F(y,z)<\frac{\ep}{2}$, hence $F(x,y)<\ep.$
\end{proof}

\begin{rem}
		If we similarly define $sp_F(\mu,n,\ep)$, then the $\mu$-$\{F_n\}$-equicontinuity implies that $\{sp_F(\mu,n,\ep)\}_{n=1}^\infty$ is bounded for any $\ep>0.$
\end{rem}
The following characterization of ergodicity was proved in \cite[Theorem 5.4]{zhen}, here we give an alterative and elementary proof.
\begin{thm}\label{ergo}
	Let $(X,T)$ be a TDS and $\mu\in M(X,T)$. Then $\mu$ is ergodic if and only if $\mu\times\mu(\{(x,y)\in X\times X: F(x,y)=0\})=1.$
\end{thm}
\begin{proof}
	$\Rightarrow$:	If $\mu$ is ergodic,  then by Birkhoff ergodic theorem, there exists $X_0$ with $\mu(X_0)=1$ such that $ w(x)=\{\mu\}, \forall x\in X_0.$ For $(x,y)\in X_0\times X_0,$  by Theorem \ref{thm2}, we have  $F(x,y)=0.$
	
	$\Leftarrow$:
	Assume there exists $A\subset X$ such that $T^{-1}A=A$ and $0<\mu(A)<1$. Choose $0<\delta_0<\frac{1}{2}$ and two compact sets $K_1\subset A, K_2\subset A^c$ with $$\mu(K_1)>\mu(A)(1-\delta_0^2),\ \mu(K_2)>\mu(A^c)(1-\delta_0^2).$$ Let $\delta=d(K_1,K_2)$. Since $$\mu\times\mu(\{(x,y)\in X\times X: F(x,y)=0\})=1,$$ we can find $N\in\N$ such  that $$(\mu\times\mu)(D)>1-\mu(A)\mu(A^c)(1-\delta_0)^2,$$ where $$D=\{(x,y)\in X\times X: F_N(x,y)<(1-2\delta_0)\delta\}.$$
Since $(A,\frac{\mu}{\mu(A)},T)$ is a MPS, it is easy to see that
	$$\mu(\{x\in A:\frac{1}{N}\sum_{i=0}^{N-1}1_{K_1}(T^ix)>1-\delta_0\})>\mu(A)(1-\delta_0).$$

	Similarly $$\mu(\{x\in A^c:\frac{1}{N}\sum_{i=0}^{N-1}1_{K_2}(T^ix)>1-\delta_0\})>\mu(A^c)(1-\delta_0).$$
		Hence the intersection of  $D$ and  the product of the above two sets is not empty.

Choose $x,y$ such that $F_N(x,y)<(1-2\delta_0)\delta$ and  $$|\{0\leq i\leq N-1: T^ix \in K_1\}|>(1-\delta_0)N,$$ $$|\{0\leq i\leq N-1: T^iy \in K_2\}|>(1-\delta_0)N.$$
 By the definition of $F_N,$ there exists a permutation $\sigma$ of $\{0,1,\ldots,N-1\}$ such that $$\frac{1}{N}\sum_{i=0}^{N-1}d(T^ix,T^{\sigma(i)}y)<(1-2\delta_0)\delta.$$
But we  have
\begin{equation*}
\sum_{i=0}^{N-1}d(T^ix,T^{\sigma(i)}y)\geq\sum_{i:T^ix\in K_1,T^{\sigma(i)}y\in K_2}d(T^ix,T^{\sigma(i)}y)\geq (1-2\delta_0)N\delta.
\end{equation*}
A contradiction!
	\end{proof}
\begin{rem}
	We can see the proof of $``\Leftarrow"$ part does not rely on any tool in ergodic theory.
\end{rem}

\section{$\mu$-$FK$-equicontinuity}\label{sect:FK-equi}
  In \cite{loose}, the authors showed that the notion of $\mu$-$FK$-equicontinuity is closely related to the concept of loosely Kronecker. In this section we  study this notion and strengthen some results in \cite{loose}.

\medskip

First we  give the definition of $\mu$-$FK$-equicontinuity.
\begin{de}
	Let $(X, T)$ be a TDS and $\mu\in M(X,T)$. We
	say that $(X, T )$ is {\it $\mu$-$FK$-equicontinuous} if for every $\tau>0$ there exists a compact set
	$M\subset X$ with $\mu(M)>1-\tau$, such that for every $\epsilon>0$, there exists $\delta>0$ such that
	for every $x,y\in M$, $d(x, y) < \delta\Rightarrow d_{FK}(x, y) <\epsilon$.
\end{de}
\begin{rem}
	\begin{enumerate}
		\item $\mu$-$FK$-equicontinuity is equivalent to   $\mu$-$\{FK_n\}$-equicontinuity (please see Theorem \ref{thmA}).
		\item It is easy to see that $\mu$-$\{FK_n\}$-equicontinuity implies that $\{sp_{FK}(\mu,n,\ep)\}_{n=1}^\infty$ is bounded for every $\ep>0$, which in turn implies zero entropy by Remark \ref{rem:brin-katok:3}\eqref{rem:brin-katok:3:3}.
	\end{enumerate}

\end{rem}

\medskip

 In \cite{loose} the authors showed that when $\mu$ is ergodic, $\mu$-$\{FK_n\}$-equicontinuity is equivalent to the following stronger statement: There exists some $M$ with full measure such that for all $x,y\in M,$ $d_{FK}(x, y) = 0.$  In this case, the author proved that $\mu$-$\{FK_n\}$-equicontinuity is equivalent to loosely Kronecker. In their proof,  the following Katok's Criterion (please see \cite[Theorem 4]{katok} or \cite[section 6]{orw} for details) played a key role:
\begin{thm}\cite[Katok's Criterion]{katok}\label{ka}
 A MPS $(X,\mu,T)$ is loosely Kronecker if and only if
 the following statement	 holds:

 $(*)$ for every finite partition $\mathcal{P}$ and $\ep>0$, there exists $N$ such that for every $n>N$, there exists a word $w\in\mathcal{P}^n$ such that
	$\mu(\{w^\prime\in \mathcal{P}^n: \bar{f}_n(w,w^\prime) <\ep\}) \geq 1-\ep.$
\end{thm}

Note that there is no topology in the above statement. However, when there is a metric,  by generalizing the idea in \cite{loose}, we  give a topological version of the  Katok's Criterion.
\begin{thm}\label{thm:f-bar}
	Let $(X,T)$ be a TDS and $\mu\in M(X,T).$ Then the following statements are equivalent:
	\begin{enumerate}
		\item  $\mu\times\mu(\{(x,y)\in X\times X:d_{FK}(x, y) = 0\})=1$.
		\item There exists $M\subset X$ with $\mu(M)=1$ such that $d_{FK}(x, y) = 0,\forall x,y\in M.$
		\item 	the condition $(*)$ in Theorem \ref{ka} holds.
	\end{enumerate}
Moreover, any one of (1), (2), (3) implies the ergodicity of $\mu.$
\end{thm}

\begin{proof} (1) or (2) implies the ergodicity of $\mu$ is from Theorem \ref{ergo}, Theorem \ref{thm:FK=F} and the fact that $\tilde{d}_{FK}\leq {d}_{FK}$. (3) implies the ergodicity of $\mu$ is \cite[Corollary 9.2]{katok}.
	
		$(1)\Rightarrow(2)$: By Fubini's Theorem, there is some $x_0$ such that $\mu(B_{d_{FK}}(x_0,0))=1$, where $$B_{d_{FK}}(x_0,0)=\{y\in X:d_{FK}(x_0, y) = 0\}.$$ $B_{d_{FK}}(x_0,0)$ is the set we need.
		
	$(2)\Rightarrow(1)$: Obvious.

	$(2)\Rightarrow(3)$: See \cite[Theorem 4.5]{loose}.
	We give a proof for completeness.
	Let $\mathcal{P}$ be a finite partition  and $\ep>0.$ Assume $\mathcal{P}=\{P_1,\ldots,P_k\}$. We can find closed set $Q_i\subset P_i, i=1,\ldots,k$ such that $\mu(\bigcup_{i=1}^kQ_i)>1-\frac{\ep}{2}$. Set $K=\bigcup_{i=1}^kQ_i.$ Then $\mu(K)>1-\frac{\ep}{2}.$ Let $b=\min_{i\ne j}d(Q_i,Q_j).$ Let $\delta=\min\{b, \frac{\ep}{3}\}.$ Choose $x_0\in M.$ For $N\in\N$, let $$A_N=\{x\in M: d_{FK_n}(x,x_0)<\frac{\delta}{2}, \forall \ n>N\} .$$ It is easy to see that $A_N$'s are increasing and $\bigcup_{N=1}^\infty A_N=M.$ We can choose some  $N_1$ with $\mu(A_{N_1})>1-\frac{\ep}{2}.$  Since $$\frac{1}{n}\sum_{i=0}^{n-1}1_K(T^ix)\stackrel{a.e.}\to \mu(K)>1-\frac{\ep}{3},$$
	we can find $N_2$ such that
	$\mu(B)>1-\frac{\ep}{2},$
	where $$B=\{x\in X: \frac{1}{n}\sum_{i=0}^{n-1}1_K(T^ix)>1-\frac{\ep}{3},\ \forall\ n>N_2\}.$$
	Hence $\mu(B\cap A_{N_1})>1-\ep.$ Let $N=\max\{N_1,N_2\}.$ For every $n>N, x ,y \in B\cap A_{N_1},$ we have
	$$\frac{1}{n}\sum_{i=0}^{n-1}1_K(T^ix)>1-\frac{\ep}{3},\ \frac{1}{n}\sum_{i=0}^{n-1}1_K(T^ix)>1-\frac{\ep}{3}.$$
	Let $D_x=\{0\leq i\leq n-1:T^ix\in K \}.$
	We have $|D_x|>(1-\frac{\ep}{3})n$. Similarly define $D_y$ and $|D_y|>(1-\frac{\ep}{3})n$. Since $d_{FK_n}(x,y)<\delta$, we can find $D\subset \{0,1,\dots,n-1\}$ and an order-preserving bijection $\pi: D\to \pi(D)$ such that
	$$|D|>(1-\delta)n>(1-\frac{\ep}{3})n$$ and $$d(T^ix, T^{\pi(i)}y)<\delta<b,\ \forall i\in D.$$
	Let $$D^\prime=\pi^{-1}(\pi(D\cap D_x)\cap D_y).$$
	Then $|D^\prime|>(1-\ep)n.$ For $i\in D^\prime,$ we have $T^ix\in K, T^{\pi(i)}y\in K$ and  $d(T^ix, T^{\pi(i)}y)<b.$ Hence $\mathcal{P}(T^ix)=\mathcal{P}(T^{\pi(i)}y).$ That is $\bar{f}_n(\mathcal{P}^n(x),\mathcal{P}^n(y))<\ep.$

	$(3)\Rightarrow(2):$
	{\bf Claim:} For any $\ep>0$, there exists $\mu(K)>0$ such that $x,y\in K\Rightarrow d_{FK}(x,y)<\ep.$
	
	Proof of claim:	 Let $\ep>0.$ Let $\mathcal{P}$ be a finite partition of $X$ with $\diam(\mathcal{P})<\ep.$ Choose $N$ such that there exists a word $w\in\mathcal{P}^N$ such that
	$$\mu(\{w^\prime\in \mathcal{P}^N: \bar{f}_N(w,w^\prime) <\frac{\ep}{2}\}) > 1-\ep.$$ Let $$K=\{w^\prime\in \mathcal{P}^N: \bar{f}_N(w,w^\prime) <\frac{\ep}{2}\}.$$ Then $$x,y\in K\Rightarrow  \bar{f}_N(\mathcal{P}^N(x),\mathcal{P}^N(y)) <\ep.$$
	Let $x,y\in K.$ Since $\mu$ is ergodic, we can assume
	$$\frac{|\{1\leq i\leq n:T^ix\in K\}|}{n}\to \mu(K),$$
	$$\frac{|\{1\leq i\leq n:T^iy\in K\}|}{n}\to \mu(K).$$
	Hence there exists $N_1$ such that for any $n>N_1,$
	 $$\frac{|\{1\leq i\leq n:T^ix\in K^c\}|}{n}<\ep,$$
	 $$\frac{|\{1\leq i\leq n:T^iy\in K^c\}|}{n}<\ep.$$
	Let $$l_0=0,\ l_1=\min\{l>N-1:T^lx\in K\},\ldots,$$
	$$ l_k=\min\{l>l_{k-1}+N-1:T^lx\in K\}\ldots,$$
	and
	$$j_0=0,\ j_1=\min\{j>N-1:T^jy\in K\},\ldots,$$
	$$ j_k=\min\{j>j_{k-1}+N-1:T^jy\in K\}\ldots.$$
	Since $T^{l_k}x,T^{j_k}y\in K,$ we have $$\bar{f}_N(\mathcal{P}^N(T^{l_k}x),\mathcal{P}^N(T^{j_k}y)) <\ep.$$
	 Hence there exists $D_k\subset [l_k,l_k+N-1]$ and an order-preserving bijection $$\pi_k:D_k\to [j_k,j_k+N-1]$$ such that $$|D_k|>N(1-\ep),
	  \ d(T^ix,T^{\pi_k(i)}y)<\ep,\ i\in D_k.$$
	Let $n>\max\{N,N_1\}$ and $k_0=\min\{k: l_k+N-1>n\}.$ Since $$[l_i+N,l_{i+1}-1]\subset \{i\in\N:T^ix\in K^c\}$$ and
	$$|\{1\leq i\leq n:T^ix\in K^c\}|<n\ep,$$
	it is easy to see that $$(k_0+1)N>n(1-\ep).$$ Similarly, let $$k_0^\prime=\min\{k: j_k+N-1>n\}.$$ Then $$(k_0^\prime+1)N>n(1-\ep).$$
	Let $$D=\bigcup_{k=0}^{[\frac{n(1-\ep)}{N}]-1}D_k$$
	and define $\pi$ on $D$ as  $\pi(i)=\pi_k(i)$ for $i\in D_k.$
Then	$$|D|=\sum |D_k|>[\frac{n(1-\ep)}{N}]N(1-\ep).$$ Hence $d_{FK_n}(x,y)<2\ep+\frac{N}{n}$ and $$d_{FK}(x,y)=\limsup_{n\to+\infty} d_{FK_n}(x,y)\leq2\ep.$$
	The proof of claim is finished.
	
	\medskip
	
Let $\ep>0$. We find $M_\ep$ as follows. Let $K$ be the set in the claim for $\ep$.	Choose $x_0\in K$.
	It is easy to see that $K\subset B_{d_{FK}}(x_0,\ep).$ Hence $\mu(B_{d_{FK}}(x_0,\ep))>0.$ Note that $B_{d_{FK}}(x_0,\ep)$ is $T$-invariant. By the ergodicity of $\mu$ we have $\mu(B_{d_{FK}}(x_0,\ep))=1.$ Let $M_\ep=B_{d_{FK}}(x_0,\ep).$ Then $\bigcap_{n=1}^\infty M_{\frac{1}{n}}$ is the set we need.
\end{proof}
\begin{rem}\label{rem:fk-f-bar}
\begin{enumerate}
  \item Theorem \ref{thm:f-bar} indicates that for the sufficiency proofs of Proposition 3.14 and Theorem 4.5 of \cite{loose} the ergodicity condition is no need to emphasize separately. We also note that the equivalence of $(2)$ and $(3)$ in Theorem \ref{thm:f-bar} is in fact  equivalent to the profound Theorem 4.5 of  \cite{loose}. Nonetheless,  our proof of $(2)\Rightarrow(3)$  seems simpler and  $(3)\Rightarrow(2)$ is new.

  \item Combining  \cite[Theorem 3.1]{jin} and Theorem \ref{ergo}, it is easy to show that for a TDS $(X,T)$ and $\mu\in M(X,T)$, the following statements are equivalent:
	\begin{enumerate}
		\item $\mu\times\mu(\{(x,y)\in X\times X: \bar{d}(x,y)=0\})=1.$
		\item $(X,\mu,T)$ is a trivial system.
\end{enumerate}	
	\end{enumerate}
\end{rem}

\medskip

In general $\mu$-$\{FK_n\}$-equicontinuity doesn't imply the ergodicity of $\mu$. Recently, Huang et al. proved the followong theorem:
\begin{thm}\label{12}\cite{xu}
	Let  $(X,T)$ be a TDS and $\mu\in M(X,T)$. Then $(X,T)$ is $\mu$-mean-equicontinuous if and only if $(X,\mu,T)$ has discrete spectrum.
\end{thm} This useful fact  helps us study $\mu$-$\{FK_n\}$-equicontinuity in non-ergodic case.

First we need some preparations.

\medskip

Let $(X,\mu,T)$ be a MPS and $A\subset X$ with $\mu(A)>0$. We can define the {\it induced system} $(A,\mu_A,T_A)$ as $\mu_A(B)=\frac{\mu(B)}{\mu(A)}$ for $ B\subset A$ and $T_A(x)=T^{r_A}(x)$ for $x\in A,$ where $$r_A(x)=\min\{n\geq 1:T^nx\in A\}.$$
Please see \cite{kaku,orw,book} for details.

We have the following theorem.
\begin{thm}\label{thm:discrete-specturm}
	Let $(X,T)$ be a TDS and $\mu\in M(X,T)$. If for any $\ep>0$ there exists a subset $A\subset X$ with $\mu(A)>1-\ep$ such that $(A,\mu_A,T_A)$ has discrete spectrum, then $(X,\mu,T)$ is $\mu$-$FK$-equicontinuous.
\end{thm}
\begin{proof}
	Let $\ep,\tau>0$ and $\tau_0=\frac{1}{2}\min\{\ep,\tau\}$. By assumption there is some $A\subset X$ with $\mu(A)>1-\tau_0^2$ such that $(A,\mu_A,T_A)$ has discrete spectrum. Hence by Theorem \ref{12}, $(A,\mu_A,T_A)$ is $\mu_A$-mean-equicontinuous. According to Theorem \ref{thmA}, $(A,\mu_A,T_A)$ is $\mu_A$-$\{\bar{d}_n\}$-equicontinuous, hence we can choose  $K\subset A$ with $\mu_A(K)>\frac{1}{\mu(A)}(1-\tau_0)$ such that
	$K$ is $T_A$-$\{\bar{d}_n\}$-equicontinuous. Hence $\mu(K)>1-\tau_0$. It is easy to see that
	$$\mu(\{x\in X:\lim_{n\to\infty}\frac{1}{n}\sum_{i=1}^{n} \chi_{A}(T^ix)>1-\tau_0\})>1-\tau_0.$$
	Let $$K_0=K\cap \{x\in X:\lim_{n\to\infty}\frac{1}{n}\sum_{i=1}^{n} \chi_{A}(T^ix)>1-\tau_0\}.$$ Then $\mu(K_0)>1-\tau$.
	
	Since $K$ is $T_A$-$\{\bar{d}_n\}$-equicontinuous, there exists $\delta>0$ such that
	$$x,y\in K,d(x,y)<\delta\Rightarrow |\{1\leq i\leq n:d(T_A^ix,T_A^iy)>\ep\}|<n\ep, \forall
	n\in\N.$$
	Let $x,y\in K_0$ and $d(x,y)<\delta.$ Since
	$$\lim_{n\to+\infty}\frac{1}{n}\sum_{i=1}^{n} \chi_{A}(T^ix)>1-\ep
	,\lim_{n\to+\infty}\frac{1}{n}\sum_{i=1}^{n} \chi_{A}(T^iy)>1-\ep,$$
	there exists $N$ such that for any $n>N$, we have
	$$\frac{1}{n}\sum_{i=1}^{n} \chi_{A}(T^ix)>1-\ep
	,\frac{1}{n}\sum_{i=1}^{n} \chi_{A}(T^iy)>1-\ep.$$
	For $n>N$, let $$\{1\leq i\leq n:T^ix\in A\}=\{i_1<\ldots<i_{k_x}\},$$
	$$\{1\leq i\leq n:T^iy\in A\}=\{j_1<\ldots<j_{k_y}\},$$ where $k_x,k_y>n(1-\ep).$
	Since $x,y\in K,d(x,y)<\delta,$ we have
	\begin{equation*}
	\begin{split}
	&|\{1\leq t\leq [n(1-\ep)]:d(T_A^tx,T_A^ty)\leq\ep\}|\\
	=&|\{1\leq t\leq [n(1-\ep)]:d(T^{i_t}x,T^{j_t}y)\leq\ep\}|
	\geq[n(1-\ep)](1-\ep).
	\end{split}
	\end{equation*}
	Let $$\{1\leq t\leq [n(1-\ep)]:d(T^{i_t}x,T^{j_t}y)\leq\ep\}=\{t_1<\ldots<t_m\}, $$
	$m\geq[n(1-\ep)](1-\ep)$	and $$D=\{i_{t_1}<\ldots<i_{t_m}\}, \pi(i_{t_k})=j_{t_k}.$$
	It is easy to see $d_{FK_n}(x,y)<2\ep+\frac{1}{n}.$ Hence $d_{FK}(x,y)\leq2\ep.$
	
	Now we prove that: for any $\ep,\tau>0$, there exist $K_0\subset X$ with $\mu(K_0)>1-\tau$, and $\delta>0$ such that $$x,y\in K_0,d(x,y)<\delta\Rightarrow d_{FK}(x,y)\leq2\ep.$$
	It follows from Proposition \ref{pm2} that $(X,\mu,T)$ is $\mu$-$FK$-equicontinuous.
\end{proof}

\begin{cor}\label{cor:discrete-specturm}
Let $(X,T)$ be a TDS with $\mu\in M(X,T)$. If $(X,\mu,T)$ is loosely Kronecker then it is $\mu$-$FK$-equicontinuous.
\end{cor}
\begin{proof}
   	If $(X,\mu,T)$ is loosely Kronecker, then the condition in Theorem \ref{thm:discrete-specturm} is satisfied. See \cite[Corollary 5.1]{katok} or \cite[Corollary 5.6]{orw} for details. Then applying Theorem \ref{thm:discrete-specturm} the corollary follows.
\end{proof}

\begin{rem}
Corollary \ref{cor:discrete-specturm} provides another new proof of the sufficiency of \cite[Theorem 4.5]{loose}.
\end{rem}

\appendix
\section{}
 In this section we use the method in \cite{xu}
and generalize some results.

 Let $(X,T)$ be a TDS and $\mu\in M(X,T)$. Let $E$ denotes the metric $d_{FK}$, $F$ or $\bar{d}$.

 We define {\em $\mu$-$\{E_n\}$-equicontinuous} as : for every $\tau>0$ there exists
 $K\subset X$ with $\mu(K)>1-\tau$, such that for every $\epsilon>0$, there exists $\delta>0$ such that
$$x,y\in K, d(x, y) < \delta\Rightarrow E_n(x, y) <\epsilon,\ \forall n\in\N.$$

 We define {\em $\mu$-$E$-equicontinuous} as : for every $\tau>0$ there exists
 $K\subset X$ with $\mu(K)>1-\tau$, such that for every $\epsilon>0$, there exists $\delta>0$ such that
 $$x,y\in K, d(x, y) < \delta\Rightarrow E(x, y) <\epsilon.$$
\begin{prop}\label{pm}
Let $(X,T)$ be a TDS and $\mu\in M(X,T)$. 	Then the following statements are equivalent:
	\begin{enumerate}
		\item $(X,T)$ is $\mu$-$\{E_n\}$-equicontinuous.
		\item For any $\tau>0$ and $\epsilon>0$, there exist subset $K$ of $X$ with $\mu(K)>1-\tau$ and $\delta>0$ such that
		\begin{equation*}
		x,y\in K,d(x,y)<\delta\Rightarrow E_n(x,y)<\epsilon,\ \forall n\in\N.
		\end{equation*}	
	\end{enumerate}
\end{prop}
\begin{proof}
	(1)$\Rightarrow$(2): It is obvious from definition.
	
	(2)$\Rightarrow$(1): Given $\tau>0$. For any $l\in\N$, by (2) there are  $K_l$  with $\mu(K_l)>1-\frac{\tau}{2^l}$ and $\delta_l>0$ such that for any $x,y\in K_l$ with $d(x,y)<\delta_l$, we have
	$E_n(x,y)<\frac{1}{l}$
	for all $n\in\N$. Let $K=\bigcap\limits_{l=1}^\infty K_l$, then $\mu(K)>1-\tau$.  It is easy to see that $K$ is the set we need.
\end{proof}
Similarly we can prove:
\begin{prop}\label{pm2}
	The following statements are equivalent:
	\begin{enumerate}
		\item $(X,T)$ is $\mu$-$E$-equicontinuous.
		\item For any $\tau>0$ and $\epsilon>0$, there exist a subset $K$ of $X$ with $\mu(K)>1-\tau$ and $\delta>0$ such that
		\begin{equation*}
		x,y\in K,d(x,y)<\delta\Rightarrow E(x,y)<\epsilon.
		\end{equation*}	
	\end{enumerate}
\end{prop}
\begin{thm}\label{thmA}
	$(X,T)$ is $\mu$-$E$-equicontinuous if and only if $(X,T)$ is $\mu$-$\{E_n\}$-equicontinuous.
\end{thm}
\begin{proof}
	Assume 	$(X,T)$ is $\mu$-$E$-equicontinuous.
	Fix $\tau>0$ and $\epsilon>0$. There exists a  set $K_0$ with $\mu(K_0)>1-\dfrac{\tau}{2}$ such that $K_0$ is $E$-equicontinuous. Then there exists $\delta_0>0$ such that for all $x,y\in K_0$ with $d(x,y)<\delta_0$, we have
	\begin{equation*}
	\limsup_{n\rightarrow\infty}E_n(x,y)<\frac{\epsilon}{2},
	\end{equation*}
	By the compactness of $X$, there exist $X_1,\ldots,X_m$ with $\textrm{diam}(X_i)<\delta_0,1\leq i\leq m,\ m\in\N$ such that $K_0=\bigcup\limits_{i=1}^mX_i$. Without loss of generality, assume $X_i\neq \emptyset$ and choose $x_i\in X_i,1\leq i\leq m.$ For $1\leq j\leq m$ and $N\in\N$, let
	\begin{equation*}
	A_N(x_j)=\{y\in X_j:{E_n}(x_j,y)<\frac{\epsilon}{2},\forall n>N\}.
	\end{equation*}
	It is easy to see that for each $1\leq j\leq m$, $\{A_N(x_j)\}_{N=1}^\infty$ is increasing and
	\begin{equation*}
	X_j=\bigcup\limits_{N=1}^\infty A_N(x_j).
	\end{equation*}
	Choose $N_0\in\N$ such that
	\begin{equation*}
	\mu(\bigcup\limits_{j=1}^mA_{N_0}(x_j))>1-\dfrac{\tau}{2}.
	\end{equation*}		
	
	By the regularity of $\mu$, we can find  pairwise disjoint compact sets
	$K_j\subset A_{N_0}(x_j),1\leq j\leq m$
	such that
	\begin{equation*}
	\mu(\bigcup_{j=1}^mK_j)>1-\tau.
	\end{equation*}
	Let $K=\bigcup\limits_{j=1}^mK_j$, then  $\mu(K)>1-\tau$.
	Choose $\delta_1>0$ such that
	$$d(x,y)<\delta_1\Rightarrow d(T^ix,T^iy)<\frac{\ep}{2}, 1\leq i\leq {N_0}.$$
	Choose $\delta>0$ with
	\begin{equation*}
	\delta<\min\{\delta_1,d(K_i,K_j),1\leq i\neq j\leq m\}.
	\end{equation*}
	For any $x,y\in K$ and $d(x,y)<\delta$, if $1\leq n\leq {N_0}$, it is easy to see that $E_n(x,y)<\ep$; if $n>{N_0}$, since $d(x,y)<\delta$, there exists $1\leq j_0\leq m$ such that $x,y\in K_{j_0}$. Then
	\begin{equation*}
	E_n(x_{j_0},x)<\frac{\epsilon}{2},\
	E_n(x_{j_0},y)<\frac{\epsilon}{2}.
	\end{equation*}
	Hence
	$E_n(x,y)<\ep.$
	By Proposition \ref{pm}, $(X,T)$ is $\mu$-$\{E_n\}$-equicontinuous.
\end{proof}

\medskip

\begin{proof}[Proof of Proposition \ref{prop:measure-complexity}]
	It is sufficient to show that if the measure FK-complexity of $(X,d,T,\mu)$  is
	weaker than $U(n)$, then  the measure FK-complexity of $(Y,d',S,\nu)$  is weaker than $U(n)$.
	
	Assume that the measure FK-complexity of $(X,d,T,\mu)$  is
	weaker than $U(n)$.
	Since $(X,\mathcal{B}(X),T,\mu)$ is measurably isomorphic to $(Y,\mathcal{B}(Y),S,\nu)$, there are
	$X'\in \mathcal{B}(X), Y' \in \mathcal{B}(Y)$ with $\mu(X') = 1$, $\nu(Y') = 1$, $TX'\subseteq X'$ and $SY'\subseteq Y'$,
	and an invertible measure-preserving map $\phi : X' \rightarrow Y'$ with $\phi \circ T(x) =
	S \circ \phi (x)$ for every $x\in X'$.

	Take $\epsilon>0$. By Lusin theorem there is a compact subset $A$ of $X'$ such that
	$\mu(A)>1-{\ep}/{2}$ and  $\phi|_A$ is a continuous function.
	Choose $\delta\in (0,\epsilon/3)$ such that
	\begin{equation*}\label{hz-1}\tag{*}
	d'(\phi(x),\phi(y))<\ep\ \text{for any}\ x,y\in A\ \text{with}\ d(x,y)<\delta.
	\end{equation*}
	
	{\bf Claim}: $sp_{FK}(d,\mu,n, \frac{\delta}{2})\ge sp_{FK}(d',\nu,n,\epsilon)$ for all $n\in \mathbb{N}$.
	
	Proof of claim:
	Fix $n\in \mathbb{N}$. Let
	$$E_n=\{x\in A:\frac{|\{0\le i \le n-1: T^ix\in A\}|}{n}\le 1-\frac{\ep}{3}\}.$$
	It is easy to see that $\mu(E_n)< \frac{\ep}{3}$. Put $A'=:A\setminus E_n$. We have $\mu(A')>1-\frac{\ep}{2}-\frac{\epsilon}{3}=1-\frac{5\ep}{6}$.
	
	For $x,y\in A'$, if $d_{FK_n}(x,y)<\delta$, then there are $D\subset \{0,1,\ldots,n-1\}$ and an order-presvering bijection $\pi\colon D\to \pi(D)\subset \{0,1,\ldots,n-1\}$ such that $|D|=|\pi(D)|>(1-\delta)n$ and
	$$
	d(T^ix, T^{\pi(i)}y)<\delta, \ \forall i\in D.
	$$
	Let $D_x=\{0\le i \le n-1: T^ix\in A\}$, $D_y=\{0\le i \le n-1: T^iy\in A\}$ and
	$$
	D'=\pi^{-1}(\pi(D\cap D_x)\cap D_y).
	$$
	Since $x,y\in A'$, then $|D'|>(1-\frac{2\ep}{3}-\delta)n>(1-\ep)n$. For $i\in D'$, we have $T^ix\in A, T^{\pi(i)}y\in A$ and $d(T^ix, T^{\pi(i)}y)<\delta$. By \eqref{hz-1} we have
	$$d'(S^i\phi(x),S^{\pi(i)}\phi(y))=d'(\phi(T^ix),\phi(T^{\pi(i)}y))<\ep$$
	for all $i\in D'$. This implies that
	$d'_{FK_{n}}(\phi(x),\phi(y))<\ep$.
	
	Pick $x_1,x_2,\cdots,x_{m}\in X$ such that $m:=m(n)=sp_{FK}(d,\mu,n,\frac{\delta}{2})$ and
	$$
	\mu\big(\bigcup_{i=1}^{m} B_{d_{FK_n}}(x_i,\frac{\delta}{2})\big)>1-\frac{\delta}{2}.
	$$
	Let $I_n=\{ r\in [1,m]: B_{d_{FK_n}}(x_r,\frac{\delta}{2})\cap A' \neq \emptyset\}$.
	For $r\in I_n$, we choose $y_r^n\in B_{d_{FK_n}}(x_r,\frac{\delta}{2})\cap A'$. Then
	$$\bigcup_{r\in I_n} \big( B_{d_{FK_n}}(y_r^n,\delta)\cap A'\big) \supseteq \bigcup_{r\in I_n} \big( B_{d_{FK_n}}(x_r,\frac{\delta}{2})\cap A' \big)=\big(\bigcup_{i=1}^m B_{d_{FK_n}}(x_i,\frac{\delta}{2})\big)\cap A'.$$
	Thus
	$$\mu(\bigcup_{r\in I_n} \big(B_{d_{FK_n}}(y_r^n,\delta)\cap A'\big))\ge \mu(\big(\bigcup_{i=1}^m B_{d_{FK_n}}(x_i,\frac{\delta}{2})\big)\cap A')
	>1-\frac{\delta}{2}-\frac{5\ep}{6}> 1-\epsilon.$$
	Since $d'_{FK_n}(\phi(x),\phi(y))<\ep$ for $x,y \in A'$ with $d_{FK_n}(x,y) < \delta$,
	one has $$\phi(B_{d_{FK_n}}(y_r^n,\delta)\cap A')\subseteq B_{d'_{FK_n}}(\phi(y_r^n),\epsilon)$$ for $r\in I_n$.
	Thus
	\begin{align*}
	\nu(\bigcup_{r\in I_n} B_{d'_{FK_n}}(\phi(y_r^n),\epsilon))&\ge
	\nu(\bigcup_{r\in I_n} \phi(B_{d_{FK_n}}(y_r^n,\delta)\cap A'))\\
	&=\nu (\phi \big(\bigcup_{r\in I_n}B_{d_{FK_n}}(y_r^n,\delta)\cap A'\big ))
	=\mu(\bigcup_{r\in I_n}B_{d_{FK_n}}(y_r^n,\delta)\cap A')>1-\epsilon.
	\end{align*}
	Hence $sp_{FK}(d',\nu,n, \ep)\le |I_n|\le m(n)=sp_{FK}(d,\mu,n, \frac{\delta}{2})$. The proof of claim is finished.
	
	\medskip
	Now since the measure FK-complexity of $(X,d,T,\mu)$  is weaker than $U(n)$ in the sense that
	$\liminf_{n\rightarrow +\infty} \frac{sp_{FK}(d,\mu,n, \frac{\delta}{2})}{U(n)}=0$, then we have $\liminf_{n\rightarrow +\infty} \frac{sp_{FK}(d',\nu,n, \ep)}{U(n)}=0$.
	This finishes the proof.
\end{proof}	

\begin{rem}
	By Proposition \ref{prop:measure-complexity},
	the measure FK-complexity of $(X,d,T,\mu)$ is weaker than $U(n)$ if and only if the measure FK-complexity of
	$(X,d',T,\mu)$ is also weaker than $U(n)$ for any compatible metric $d'$ on $X$.
	Thus we can simply say that the measure FK-complexity of $(X,T,\mu)$ is weaker than $U(n)$.
\end{rem}

\end{document}